\documentclass[a4paper,11pt]{amsart}

\usepackage{a4wide}
\usepackage{amsmath}
\usepackage{amsfonts}
\usepackage{amssymb}
\usepackage{amsthm}
\usepackage[hyperfootnotes=false]{hyperref} 

\allowdisplaybreaks[3]

\numberwithin{equation}{section}

\newtheorem{theorem}{Theorem}[section]
\newtheorem{lemma}[theorem]{Lemma}

\newtheorem{remark}[theorem]{Remark}
\newtheorem{proposition}[theorem]{Proposition}
\newtheorem{definition}[theorem]{Definition}

\newcommand{\dd}{\,\mathrm{d}}
\newcommand{\var}{\text{-}\mathrm{var}}

\newcommand{\R}{\mathbb{R}}
\newcommand{\N}{\mathbb{N}}

\newcommand{\1}{\mathbf{1}}
\newcommand{\cal}{\mathcal}
\renewcommand{\d}{\mathrm{d}}

\title{Characterization of non-linear Besov spaces}

\author[Liu]{Chong Liu}
\address{Chong Liu, Eidgen\"ossische Technische Hochschule Z\"urich, Switzerland}
\email{chong.liu@math.ethz.ch}

\author[Pr\"omel]{David J. Pr\"omel}
\address{David J. Pr\"omel, University of Oxford, United Kingdom}
\email{proemel@maths.ox.ac.uk}

\author[Teichmann]{Josef Teichmann}
\address{Josef Teichmann, Eidgen\"ossische Technische Hochschule Z\"urich, Switzerland}
\email{josef.teichmann@math.ethz.ch}

\date{\today}

\begin{document}

\begin{abstract}
  The canonical generalizations of two classical norms on Besov spaces are shown to be equivalent even in the case of non-linear Besov spaces, that is, function spaces consisting of functions taking values in a metric space and equipped with some Besov-type topology. The proofs are based on atomic decomposition techniques and metric embeddings. Additionally, we provide embedding results showing how non-linear Besov spaces embed into non-linear $p$-variation spaces and vice versa. We emphasize that we neither assume the UMD property of the involved spaces nor their separability.
\end{abstract}

\maketitle 

$p$-variation, fractional Sobolev space. 


\section{Introduction}

There are various ways to define Besov spaces consisting of functions with values in the real numbers or even Banach spaces, see, e.g., the introductory books~\cite{triebel2010}, \cite{Leoni2017} or \cite{Sawano2018}. In many settings these definitions are shown to be equivalent but often under the assumption of additional properties of the target Banach spaces like the unconditional martingale difference (UMD) property or separability. Maybe most important in theory as well as in applications are equivalent characterizations of the $B^s_{p,q}$-Besov regularity of a function in a countable manner like by a respective Besov sequence space~$b^s_{p,q}$. The simplest countable representation is to evaluate a Besov regular function on a countable set, for instance, on the set of dyadic points in an interval. Frequently, the sequences in~$b^s_{p,q}$ have an interpretation as coefficients of basis expansions with respect to some wavelet bases or splines. This usually provides isomporphisms between Besov spaces and sequence spaces and very precise assertions of the expansion coefficients and embedding theorems.

In this article we consider Besov spaces~$B^s_{p,q}$ consisting of functions $f\colon [0,1]\to E$ with values in a general (non-linear) metric space~$(E,d)$. Note that the classical definition of Besov spaces in terms of integrals has a canonical extension to the metric setting, cf.~\eqref{eq: Besov function in metric spaces} below. While there are several other equivalent characterizations of scalar valued Besov spaces, in particular, in terms of atomic decomposition or of coefficients of wavelets expansions, most of them seem to have no direct interpretation for functions taking values in general metric spaces. However, we provide an important equivalent characterization of the $B^s_{p,q}$-Besov property of functions on the unit interval by the sequence of their values on dyadic points (and an additional continuity property) in the metric setting, cf.~\eqref{eq: discete Besov function in metric spaces} below. So far this equivalence we are dealing with is only known in the case of real valued functions due to several authors, see, e.g., the works of Kamont~\cite{Kamont1997}, of Bodin~\cite{Bodin2009} or of Rosenbaum~\cite{Rosenbaum2009}. But up to now, it was an open question whether this equivalence holds beyond finite dimensional target spaces. 

We shall provide first a general proof for Besov regular functions with values in general Banach spaces. In this case we actually obtain the equivalence between three differently defined Besov norms: the most classical one based on integrals, the one using second order differences and the one relying on first order differences on the dyadic points. By metric embedding results we then extend the equivalence to metric space valued curves for those two characterizations of Besov regularity which have canonical meanings for metric spaces. This extension argument crucially requires the equivalence of the Besov norms for functions taking values in general Banach spaces without assuming any additional property of the target Banach spaces like the UMD property.
 
It is remarkable that by linear methods like atomic decompositions, cf.~\cite{Scharf2012}, or wavelet expansions, cf.~\cite{Triebel2004}, we are able to prove a non-linear result for curves taking values in a metric space. Additionally, let us remark that we could only prove the result by abstract tensorization techniques for Besov functions taking values in nuclear spaces, which does not help for the general case since no sufficiently strong embedding results into nuclear spaces exist: therefore abstract tensorization is unfortunately not useful here and we had to look for more direct approaches. Furthermore, notice that our proof provides many more interesting discrete characterizations depending on the applied atomic or wavelet representation.

The motivation to consider functions from the interval $[0,1]$ to a metric space~$E$ comes from potential applications in stochastic analysis, the theory of stochastic processes and data science. Indeed, in these areas one frequently deals with (random) functions from a time interval, here normalized to $[0,1]$, taking values in a possibly non-linear space. An example of such a function is a ``rough path'' in the sense of T. Lyons, which is a mapping from an interval into the free nilpotent group generated by the step-$N$ signatures, see \cite{Friz2010}. 

In the last part of this article we present embedding results showing how Besov spaces embed into $p$-variation spaces and vice versa, both again consisting of functions with values in a general metric space. Note for example: if the target space is a complete metric group with left- or right-invariant metric, then our results actually allow to define a complete metric on metric group valued Besov curves and to characterize the so obtained metric space discretely via a sequence space of metric group valued sequences. In combination with our Besov--p-variation embedding results, this allow to derive fundamental results of rough path theory in a Besov space setting such as the continuity of the solution map of rough differential equations~\cite{Lyons1998} or Lyons--Victoir's extension theorem~\cite{Lyons2007a}, cf.~\cite{Liu2018}, which were classically proven in H\"older-type or $p$-variation distances. Here the discrete characterization of non-linear Besov spaces seems to be crucial and, moreover, Besov norms are particularly useful due to their smoothness properties, for instance, to construct unique rough path extensions with minimal Besov norms, see \cite{Liu2018}. Hence, the discrete characterization of non-linear Besov space paves the way for a novel full-fledged Besov theory of rough paths. 

Further possible applications include delicate regularity questions for (infinite dimensional space valued) stochastic processes, cf.~\cite{Kamont1997b} or~\cite{Veraar2009}, regularity of data streams, cf. \cite{Rosenbaum2011}, and applications to Besov versions of the theory of regularity structures, cf.~\cite{Hairer2017}.

\smallskip
\noindent{\bf Organization of the paper:} In Section~\ref{sec:main results} we present the main results and the considered function spaces are introduced. Section~\ref{sec:Besov norms} provides the equivalence of three classical norms on vector-valued Besov spaces. The embedding results between Besov spaces and $p$-variation spaces are proven in Section~\ref{sec:embedding results}.

\smallskip
\noindent{\bf Acknowledgment:} Chong Liu and Josef Teichmann gratefully acknowledge support by the ETH foundation. Josef Teichmann gratefully acknowledges support from SNF project 163425.

\section{Main results and function spaces}\label{sec:main results}

Let $(E,d)$ be a metric space. For $0<s<1$ and $1\leq p,q\leq \infty$, two versions of Besov spaces consisting of functions with values in the metric space~$E$ can be defined as follows.
\begin{itemize}
  \item $B_{p,q}^s([0,1];E)$ is the space of all measurable functions $f\colon [0,1] \to E $ such that 
        \begin{equation}\label{eq: Besov function in metric spaces}
          {\|f\|}_{s,p,q} :=\bigg(\int_{0}^1 \Big(\int_0^{1-h} \frac{d(f(x),f(x+h))^p}{h^{sp}}\dd x\Big)^{\frac{q}{p}}\frac{\dd h}{h} \bigg)^{1/q}<\infty.
        \end{equation}
  \item $b_{p,q,(1)}^s([0,1];E)$ is the space of all \textit{continuous} functions $f\colon [0,1] \to E $ such that
        \begin{equation}\label{eq: discete Besov function in metric spaces}
          {\|f\|}_{s,p,q;(1)} :=\bigg( \sum_{j \geq 0} 2^{jq(s - \frac{1}{p})} \Big(\sum_{m=0}^{2^j-1} { d\big(f\big(\frac{m+1}{2^j}),f\big(\frac{m}{2^j}\big) \big)  }^p\Big)^{\frac{q}{p}}\bigg)^{1/q} <\infty.
        \end{equation}
\end{itemize}
In the case of $p=\infty$ or $q=\infty$ we use the standard modifications of~\eqref{eq: Besov function in metric spaces} and~\eqref{eq: discete Besov function in metric spaces}. The space $B^s_{p,q}([0,1];E)$ is called ($E$-valued) Besov space and an element $f\in B^s_{p,q}([0,1];E)$ is said to be a ($E$-valued) Besov function.

The Besov spaces~$B^s_{p,q}([0,1];E)$ cover many well-known function spaces as special cases. Namely, the space $B^s_{\infty,\infty}([0,1];E)=:C^{s}([0,1];E)$ is the space of H\"older continuous functions, the space $B^s_{p,\infty}([0,1];E)$ is sometimes called Nikolskii space, and $B^s_{p,p}([0,1];E)$ is referred to as (fractional) Sobolev space. Note, that for $p=q$ the quantity~\eqref{eq: Besov function in metric spaces} is equivalent to 
\[
  \Big(\int_0^1 \int_0^1 \frac{d(f(s),f(t))^p}{|t-s|^{sp+1}} \dd s \dd t\Big)^{\frac{1}{p}},
\]
which usually serves as defining property of fractional Sobolev spaces, see~\cite[Example~5.16]{Friz2010} or \cite{Simon1990}. For more comprehensive introductions to these function spaces we refer, for instance, to \cite{triebel2010} or \cite{Peetre1976}. Furthermore, the space of continuous functions $f\colon [0,1] \to E$ is denoted by $C([0,1];E)$.

\begin{remark}\label{rmk:besov embedding}
  In the present work we focus on the parameters $s \in (0,1)$ and $p,q\in [1,\infty]$ with $s > 1/p$. Under these conditions, an $E$-valued Besov function $ f\in B_{p,q}^s([0,1];E)$ is immediately a continuous function, see e.g.~\cite[Corollary~26]{Simon1990}, and an application of the Garcia--Rodemich--Rumsey inequality, see e.g.~\cite[Theorem~A.1]{Friz2010}, implies furthermore the existence of a constant $C > 0$ such that
  \[
    d(f(t),f(s)) \leq C {|t-s|}^{s-\frac{1}{p}} 
  \]
for all $s,t \in [0,1]$ and all $f \in B^s_{p,q}([0,1];E)$. However, for a function $f\in b_{p,q,(1)}^s([0,1];E)$ the continuity is an additionally necessary assumption because the quantity ${\|f\|}_{s,p,q;(1)}$ only sees the function~$f$ evaluated at a countable subset of the interval~$[0,1]$.
\end{remark}

The main contribution of the present article is to show that the discrete characterization~\eqref{eq: discete Besov function in metric spaces} and the integral characterization~\eqref{eq: Besov function in metric spaces} are equivalent even in the case of Besov spaces consisting of functions taking values in a metric space. The precise statement is formulated in the next theorem. It follows from Theorem~\ref{thm: equivalence of three norms on Besov spaces} combined with a metric embedding (Subsection~\ref{subsec:embeddings of metric spaces into Banach spaces}), see Remark~\ref{rmk:proof main theorem}.

\begin{theorem}\label{thm:Besov characterization}
  Suppose that $(E,d)$ is a metric space. Let $s \in (0,1)$ and $p,q\in [1,\infty]$ be such that $s >1/p$. Then, one has
  $$
    B_{p,q}^s([0,1];E) = b_{p,q,(1)}^s([0,1];E),
  $$
  and $\|\cdot\|_{s,p,q}$ and $\|\cdot\|_{s,p,q;(1)}$ are equivalent, i.e., there exist constants $C_1,C_2 > 0$ only depending on $s$, $p$ and $q$ such that 
  $$
    C_1\|f\|_{s,p,q;(1)} \leq \|f\|_{s,p,q} \leq C_2\|f\|_{s,p,q;(1)}
  $$
  for all $f \in C([0,1];E)$.
\end{theorem}

\begin{remark}
  Assuming~$E$ is a Banach space, a third characterization of Besov spaces is based on second order differences, see~\eqref{eq:Besov second order differences} in Subsection~\ref{subsec:vector-valued Besov spaces}. In this case, Theorem~\ref{thm: equivalence of three norms on Besov spaces} provides that all three characterizations of Besov spaces (consisting of Banach space valued functions) are equivalent. However, the characterization using second order differences seems to have no canonical generalization to metric space valued functions.  
\end{remark}

\begin{remark}
  Assuming~$E$ is the Euclidean space~$\R$, Theorem~\ref{thm:Besov characterization} and Theorem~\ref{thm: equivalence of three norms on Besov spaces} are well-known results. The equivalence stated in Theorem~\ref{thm:Besov characterization} goes back at least to the work of Kamont~\cite{Kamont1997}, where the result was proven for (anisotropic) Besov spaces on~$[0,1]^d$ and other versions can be found, e.g., in \cite{Rosenbaum2009} or \cite{Bodin2009}. The equivalence between~\eqref{eq: Besov function in metric spaces} and the Besov regularity formulated using second order differences (see~\eqref{eq:Besov second order differences}) was, for instance, considered in the paper~\cite{Ciesielski1993} by Ciesielski, Kerkyacharian and Roynette.
\end{remark}

Instead of Besov spaces, in stochastic analysis or probability theory the space of continuous functions of finite $p$-variation is more frequently used, which also possess a natural extension to the metric setting. To introduce its definition, we call $\mathcal{P}$ a partition of the interval $ [0,1]$ if $\mathcal{P}=\{[t_i,t_{i+1}]\,:\, 0=t_0 < t_1<\dots <t_n=1,\, n\in \mathbb{N}\}$. Then, for $p\in[1,\infty)$ the $p$-variation space $C^{p\var}([0,1];E)$ consists of all functions $f\in C([0,1];E)$ such that
\begin{equation*}
  \|f\|_{p\var}:=\bigg(\sup_{\mathcal{P}} \sum_{[s,t]\in \mathcal{P}} d(f(s),f(t))^p \bigg)^{\frac{1}{p}}<\infty,
\end{equation*}
where the supremum is taken over all partitions $\mathcal{P}$ of the interval~$[0,1]$. 

The next theorem relates non-linear Besov and $p$-variation spaces. It is a summary of Proposition~\ref{prop:variation embeddings} and~\ref{prop:besov embedding} extended to metric spaces by Kuratowski's embedding, see Subsection~\ref{subsec:embeddings of metric spaces into Banach spaces}.

\begin{theorem}\label{thm:embedding}
  Suppose that $(E,d)$ is a metric space. Let $s \in (0,1)$ and $p,q\in [1,\infty)$ be such that $s >1/p$. Set 
  \begin{equation*}
   \quad \beta:= \big(s + ((q^{-1}-p^{-1})\wedge 0)\big)^{-1}.  
  \end{equation*}
  Then, one has the following continuous embeddings 
  \begin{equation*}
    B^{s}_{p,q}([0,1];E)  \subset C^{\beta\var}([0,1];E) \subset B^{1/\beta}_{\beta,\infty}([0,1];E)
  \end{equation*}
  and there exist constants $C_1,C_2 > 0$ only depending on $s$, $p$ and $q$ such that 
  \begin{equation*}
   \| f \|_{1/\beta,\beta,\infty}  \leq C_1 \| f \|_{\beta \var}  \leq C_2  \| f \|_{s,p,q}
  \end{equation*}
  for all $f \in C([0,1];E)$.
\end{theorem}

\begin{remark}
  In the specific case of $E$ being the Euclidean space~$\R$, Theorem~\ref{thm:embedding} is again well-known and the necessary integral estimates to prove the embeddings can already be found in the works of Young~\cite{Young1936} and of Love and Young~\cite{Love1938}, see also e.g. \cite[Theorem~2]{Rosenbaum2009} for the second embedding. 
  
  For general metric spaces~$E$, the first embedding was proven for fractional Sobolev spaces $B^{s}_{p,p}([0,1];E)$ by Friz and Victoir in \cite[Theorem~2]{Friz2006}.
\end{remark}

\subsection{Embeddings of metric spaces into Banach spaces}\label{subsec:embeddings of metric spaces into Banach spaces}

It is due to, e.g., Kuratowski, see~\cite{Kuratowski1977}, that we can isometrically embed any metric space $ (E,d) $ into a subset of the Banach space of bounded continuous functions $ C_b(E) $, just via $ i\colon x \mapsto (y \mapsto d(x,y) - d(x_0,y)) $ (for some fixed anchoring point $x_0\in E$). There are several further generic isometric embeddings of metric spaces into Banach spaces, often spaces of functions with supremum norms, but usually it is hard to find embeddings into spaces with particular properties, like, e.g., the UMD property.

Additionally to Kuratowski's embedding, we also mention Aharoni's bi-lipschitz embedding, see~\cite{Aharoni1974}, since for our purposes it is actually enough to consider Lipschitz equivalence: every complete and separable metric space is Lipschitz equivalent to a closed subset of the sequence space 
$$
  c_0 :=\big \{(c_i)_{i \in \N} \,:\, c_i \in \R, \lim_{i \rightarrow \infty} c_i = 0\big \},
$$
that is, there exists a mapping $ T \colon E \to c_0 $ such that there 
$$
  K_1 d(x_1,x_2) \leq \| T(x_1) - T(x_2) \| \leq K_2 d(x_1,x_2), 
$$
for all $ x_1,x_2 \in E $  and for some positive constants $ 0 < K_1 < K_2 $.  

Hence, it is enough to show Theorem~\ref{thm:Besov characterization} and Theorem~\ref{thm:embedding} just for Banach spaces.

\subsection{Notation}

Let us briefly fix fairly standard notation for the sake of clarity. 

The natural numbers are denoted by $\N:=\{1,2,\dots\}$, the natural numbers including $0$ are $\N_0=\{0\} \cup \N$, $\mathbb{C}$ and $\R$ are the complex and real numbers, respectively, and $\mathbb{Z}$ stands for the set of all integers. For $x\in\R$ we set $\lfloor x\rfloor:=\sup\{y\in\mathbb{Z}\,:\, y< x\}$. For two real functions $a,b$ depending on variables $x$ one writes $a\lesssim b$  if there exists a constant $C>0$ such that $a(x) \leq C\cdot b(x)$ for all $x$, and $a\sim b$ if $a\lesssim b$ and $b\lesssim a$ hold simultaneously. 

Let $(E,\|\cdot\|)$ be a Banach space. $L^{p}(\mathbb{R};E)$ denotes the Lebesgue space of all measurable functions $f\colon \R\to E$ such that
\begin{equation*}
  \|f\|_{L^p} := \bigg( \int_{\R} \|f(x)\|^p\dd x \bigg)^{\frac{1}{p}}<\infty 
\end{equation*}
and $L^{p}((0,1);E)$ denotes the space of all measurable functions $f\colon (0,1)\to E$ such that
\begin{equation*}
  \|f\|_{L^p((0,1);E)} := \bigg( \int_{0}^1 \|f(x)\|^p\dd x \bigg)^{\frac{1}{p}}<\infty. 
\end{equation*}
The space $\mathcal{S}^{\prime}(\mathbb{R},E)$ is the space of $E$-valued tempered distributions, that is, the space of all continuous and linear mappings from the $\R$-valued Schwartz space $\mathcal{S}(\mathbb{R})$ into $E$. The Fourier transform on  
$\mathcal{S}^{\prime}(\mathbb{R},E)$ is $\mathcal{F}$ and its inverse is $\mathcal{F}^{-1}$.

On a Banach spaces~$(X,\|\cdot\|)$, two norms $\| \cdot\|_1$ and $\|\cdot\|_2$ are said to be (strongly) equivalent if there exist constants $C_1, C_2>$ such that 
\begin{equation*}
  C_1\|x\|_1\leq \|x\|_2\leq C_2 \|x\|_1, \quad \text{for all }x\in X.
\end{equation*}
We write $\|\cdot\|_1\sim \|\cdot\|_2 $ meaning $\| \cdot\|_1$ and $\|\cdot\|_2$ are equivalent norms on a Banach space~$X$.

\section{Three equivalent norms on vector-valued Besov spaces}\label{sec:Besov norms}

In this section we prove that the equivalence between three classical norms on Besov spaces extend to the Besov spaces consisting of functions taking values in a Banach space. For this purpose, we assume that $(E,\| \cdot \|)$ is a Banach space in the entire section. We start by introducing atomic decompositions and harmonic representations, which allows to describe the Besov regularity of functions.

\subsection{Atomic decompositions and harmonic representations}

As an initial step we present an expansion of continuous functions with Lipschitz atoms, relaxing the smoothness assumption on the atoms considered in the work of Scharf, Schmei\ss er and Sickel~\cite{Scharf2012}. For the reader's convenience we use the same notation and definitions as in~\cite{Scharf2012}. The theory and results are first developed for Besov spaces with domain~$\mathbb{R}$. The restriction to $[0,1]$ will be discussed below in Subsection~\ref{subsec:restriction to internval}. First let us recall the definition of vector-valued Besov spaces:

\begin{definition}
  Let $(\varphi_j)_{j \in \N_0}$ be a smooth dyadic resolution of unity. Let $1\le p,q\le \infty$ and $s \in \R$. For $f \in \mathcal{S}^\prime(\R;E)$ we define
  \begin{equation*}
    \|f\|_{\mathfrak{B}^s_{p,q}} := \Big(\sum_{j = 0}^\infty 2^{jsq} \|\mathcal{F}^{-1}(\varphi_j\mathcal{F}(f)) \|_{L^p}^q\Big)^{\frac{1}{q}}
  \end{equation*}
  and 
  \begin{equation*}
    \mathfrak{B}^s_{p,q}(\R;E) := \Big \{f \in \mathcal{S}^\prime(\R;E): \|f\|_{\mathfrak{B}^s_{p,q}} < \infty\Big\}.
  \end{equation*}
  The space~$\mathfrak{B}^s_{p,q}(\R;E)$ is called ($E$-valued) Besov space.
\end{definition}

Note that the Besov spaces with $0<p,q<1$ can be defined in the same manner as above, see \cite[Definition~2.1]{Scharf2010}. For more information about vector-valued tempered distributions and smooth dyadic resolution of unity we refer again to the paper~\cite{Scharf2010}.

If $0<s<1$, we can characterize vector-valued Besov spaces in terms of first order difference, which can be viewed as the normed space version of~\eqref{eq: Besov function in metric spaces}:

\begin{definition}
  Let $1\le p,q\le \infty$ and $s \in (0,1)$. Then $f \in B^s_{p,q}(\R;E)$ if $f \in L^p(\R;E)$ and
  \begin{equation}\label{eq:definition of Besov spaces by first difference}
    \|f |B^s_{p,q}(\mathbb{R};E)\|:= \|f\|_{L^p} + \bigg( \int_{\mathbb{R}} \Big(\int_{\mathbb{R}} \frac{\|f(x+h)-f(x)\|^p}{h^{sp}}\dd x\Big)^{\frac{q}{p}}\frac{\dd h}{h} \bigg)^{1/q}<\infty.
  \end{equation}
\end{definition}

It is a well-known result that for $s \in (0,1)$, one has $B^s_{p,q}(\R;E) = \mathfrak{B}^s_{p,q}(\R;E)$ so that $\| \cdot |B^s_{p,q}(\mathbb{R};E)\|$ and $\| \cdot \|_{\mathfrak{B}^s_{p,q}}$ are equivalent norms, see e.g. \cite{Amann1997}. Therefore, from now on we will not distinguish between $B^s_{p,q}(\R;E)$ and $\mathfrak{B}^s_{p,q}(\R;E)$, and will always use~\eqref{eq:definition of Besov spaces by first difference} as the Besov norm.

\begin{remark}
  If we additionally have $1/p < s < 1$, then the Besov space $B^s_{p,q}(\R;E)$ can be embedded into the H\"older space $C^{s-1/p}(\R;E) = B^{s-1/p}_{\infty,\infty}(\R;E)$ by the Garcia-Rodemich-Rumsey inequality as already discussed in Remark~\ref{rmk:besov embedding}. In particular, every $f \in B^s_{p,q}(\R;E)$ is a continuous function. For more details about embedding results for vector valued Besov spaces we refer to \cite{Simon1990}.
\end{remark}

A natural way to describe the regularity of functions is based on the concept of atoms. To recall this concept, denote by $Q_{j,m}:=\{x \in \mathbb{R}: |x-2^{-j}m|\leq 2^{-j-1}\}$ the interval with the center at $2^{-j}m$ and side length $2^{-j}$ for $m \in \mathbb{Z}$ and $j \in \mathbb{N}_0$. For $K \in \mathbb{N}_0$ and $d>1$, a $K$-times differentiable (in the case $K=0$ continuous) function $a\colon\mathbb{R} \rightarrow E$ is called ($E$-valued) \textit{$1$-atom (more exactly $1_K$-atom)} if $\textup{supp} \, a \subset d \cdot Q_{0,m}$   for an  $m \in \mathbb{Z}$ and $\sup_{x}\|D^{\alpha} a(x)\| \leq 1$  for all $\alpha\leq K$ and for all $x \in \R$. For $s \in \mathbb{R}$, $0<p\leq\infty$ and $L+1 \in \mathbb{N}_0$, a $K$-times differentiable (in the case $K=0$ continuous) function $a\colon\mathbb{R} \rightarrow E$ is called ($E$-valued) \textit{$(s,p)$-atom (more exactly $(s,p)_{K,L}$-atom)} if there exists a $j \in \mathbb{N}_0$ such that $\textup{supp} \ a \subset d \cdot Q_{j,m}$ for an $ m \in \mathbb{Z}$, $\sup_{x} \|D^{\alpha} a(x)\|  \leq 2^{-j (s-\frac{1}{p})+\alpha j}$ for all $|\alpha|\leq K$, and
  $\int_{\mathbb{R}} x^{\beta} a(x) \dd x=0$ for all $\beta \leq L$.
In particular, $a_{j,m}e_{j,m}$ is a vector-valued $(s,p)_{K,L}$-atom if $a_{j,m}$ is a scalar (i.e. $\mathbb{C}$-valued) $(s,p)_{K,L}$-atom and $e_{j,m} \in U_E:=\left\{x \in E: \|x\|=1\right\}$.\smallskip

Furthermore, we introduce the sequence space~$b_{p,q}$:

\begin{definition}
  Let $0<p\leq \infty$, $0<q\leq \infty$ and let $\lambda$ denote a real-valued sequence of the form
  \begin{align*}
    \lambda :=\left\{ \lambda_{j,m} \in \mathbb{R}\,:\, j \in \mathbb{N}_0, m \in \mathbb{Z} \right\}.
  \end{align*}
  The sequence space~$b_{p,q}$ is defined as
  \begin{align*}
    b_{p,q}:=\left\{\lambda: \|\lambda\, | b_{p,q}\| :=\left(\sum_{j=0}^{\infty} \left(\sum_{m \in \mathbb{Z}} |\lambda_{j,m}|^p\right)^{\frac{q}{p}} \right)^{\frac{1}{q}} <\infty \right\},
  \end{align*}
  appropriately modified in the cases $p=\infty$ or $q=\infty$.
\end{definition}

As already mentioned in Subsection~\ref{subsec:embeddings of metric spaces into Banach spaces}, for a Banach space $E$, the Besov spaces $B^s_{p,q}(\R;E)$ can be characterized via atomic representations and, as a consequence, $B^s_{p,q}(\R;E)$ is isomorphic to suitable sequence spaces.

\begin{proposition}\label{prop:HarmAtomDarst}
  Let $1 \le p,q\le \infty$, $s>0$ and $K \in \mathbb{N}_0$ with $K\geq 1+\lfloor s \rfloor$. Then $f \in {\cal S}'(\mathbb{R};E)$ belongs to $B_{p,q}^s(\mathbb{R};E)$ if and only if it can be represented by
  \begin{align*}
    f=\sum_{j \in \mathbb{N}_0} \sum_{m \in \mathbb{Z}} \lambda_{j,m} a_{j,m}(x),
  \end{align*}
  where $a_{j,m}$ are $E$-valued $1_K$-atoms (for $j=0$) or $E$-valued $(s,p)_{K,-1}$-atoms (for $j \in \mathbb{N}$) and $\lambda \in b_{p,q}$, and the convergence being in $L^p(\R;E)$.
  Furthermore, we have 
  \begin{align*}
    \|f |B^s_{p,q}(\mathbb{R};E) \| \sim \inf \|\lambda | b_{p,q}\|
  \end{align*}
  in the sense of equivalence of norms, where the infimum on the right-hand side is taken over all admissible atomic representations for $f$.
\end{proposition}

The above Proposition~\ref{prop:HarmAtomDarst} can be found in \cite[Theorem~3.7]{Scharf2012}. We refer the interested reader to \cite{Scharf2010} and \cite{Scharf2012} for a very detailed proof of the theorem. 

It is also possible to obtain a non-smooth atomic representation of Besov spaces $B^s_{p,q}(\R;E)$, which allows us to relax the assumptions about the smoothness of the atoms $a_{j,m}$ as required in Proposition~\ref{prop:HarmAtomDarst}. This turns out to be very useful for our purposes. In particular, we will use \textit{Lipschitz atoms} which are defined as follows:

\begin{definition}\label{def: Lip-atom}
  We say that a function $a\colon \R \rightarrow E$ is a ($E$-valued) Lipschitz atom (in short Lip-atom) if there is a $d > 1$ such that
  $ supp \ a\subset d \cdot Q_{j,m}$ for some $j \in \N_0$ and for some $m \in \mathbb{Z}$, $\sup_{x} \|a(x)\| \le 2^{-j(s-\frac{1}{p})}$ and $\|a(x)|\textup{Lip}\| \leq 2^{-j(s-\frac{1}{p})+j}$, where $\|a(x)|\textup{Lip}\| := \sup_{x \neq y} \frac{\|a(x) - a(y)\|}{|x - y|}$ is the Lipschitz constant of $a(x)$.
\end{definition}

As for the smooth case, $a_{j,m}(\cdot)e_{j,m}$ is a $E$-valued Lip-atom if $a_{j,m}$ is a scalar Lip-atom and $e_{j,m} \in E$ satisfies that $\|e_{j,m}\| \le 1$.

In turns out that the atomic decomposition provided in Proposition~\ref{prop:HarmAtomDarst} extends to its counterpart in terms of Lip-atoms:

\begin{proposition}\label{prop: Lip-atomic representation of Besov functions}
  Let $ 1\le p,q \le \infty$ and $0<s<1$. Then $f \in \mathcal{S}^\prime(\R;E)$ belongs to $B^s_{p,q}(\R;E)$ if and only if it can be represented as
  $$
    f = \sum_{j = 0}^\infty \sum_{m \in \mathbb{Z}} \lambda_{j,m} a_{j,m}, 
  $$
  where $a_{j,m}$ are Lip-atoms, $\lambda \in b_{p,q}$, and the convergence being in $L^p(\R;E)$. Furthermore, we have
  \begin{align*}
    \|f|B^s_{p,q}(\mathbb{R};E) \|  \sim \inf \| \lambda | b_{p,q}\|  
  \end{align*}
  in the sense of equivalence of norms, where the infimum on the right-hand side is taken over all admissible representations for~$f$.
\end{proposition}

\begin{proof}
  In view of Proposition~\ref{prop:HarmAtomDarst}, we have a smooth atomic representation for $E$-valued Besov functions, and thus we can utilize the proof of \cite[Theorem~2.6]{Schneider2013} (originally preformed for the scalar case) in combination with the observation that $E$-valued smooth $(s,p)_K$-atoms ($K \ge 1$) are $E$-valued Lip-atoms. 
\end{proof}

\begin{remark}
  In fact, it is possible to obtain Proposition~\ref{prop: Lip-atomic representation of Besov functions} for a more general vector-valued $(\sigma,p)$-atomic representation of Besov functions, where the $(\sigma,p)$-atoms for any $s <\sigma \le 1$ are defined in the sense of Definition~2.3 in \cite{Schneider2013} (but with multiplying the weight $2^{-j (s-1/p)}$). The proof follows by the same arguments as given in the proof of \cite[Theorem~2.6]{Schneider2013}, noticing that any $E$-valued smooth $(s,p)_K$-atoms ($K \ge 1$) are $E$-valued $(\sigma,p)$-atoms with $s<\sigma\le 1$.
\end{remark}

\subsection{Besov spaces on intervals}\label{subsec:restriction to internval}

The representation from Proposition~\ref{prop: Lip-atomic representation of Besov functions} of Besov function in term of Lip-atoms can be transferred to Besov functions considered on the interval~$[0,1]$, as discussed below.

Let $0<s<1$ and $1\le p,q \le \infty$. For the interval $[0,1]$, the Besov space $B^s_{p,q}([0,1];E)$ is defined as 
$$
  B^s_{p,q}((0,1);E) := \big\{f \in \mathcal{S}^\prime((0,1);E): \exists g \in B^s_{p,q}(\R;E), g|_{(0,1)} = f\big\},
$$
and 
$$
  \|f\|_{B^s_{p,q}((0,1))} := \inf \|g|B^s_{p,q}(\mathbb{R};E)\|,
$$
where the infimum is taken over all $g \in B^s_{p,q}(\R;E)$ such that $g|_{(0,1)} = f$. One important result in Besov space theory on (bounded) interval is that there exist a continuous trace and extension operator
\[
  \operatorname{re}\colon B^s_{p,q}(\mathbb{R};E) \to B^s_{p,q}((0,1);E) 
\]
and
\[
  \operatorname{ext}\colon B^s_{p,q}((0,1);E) \to B^s_{p,q}(\mathbb{R};E)
\]
such that $\operatorname{re} \circ \operatorname{ext} = \text{id}$, where $\text{id}$ is the identity on $B^s_{p,q}((0,1);E)$. While this extension theorem is usually formulated for scalar valued Besov functions, its proof naturally carries over to the vector-valued case, see e.g. \cite[Theorem~2.2]{Rychkov1999}. Since it holds also for any $s \in \R$, $0<p,q\le \infty$ and any Lipschitz domain $\Omega \subset \R^n$ for $n \ge 1$, we derive that
$$
  \|f\|_{B^s_{p,q}(0,1)} \sim \|f\|_{L^p((0,1);E)} + \bigg(\int_0^1 \Big(\int_0^{1-h} \frac{\|f(x+h) - f(x)\|^p}{h^{sp}} \dd x\Big)^{\frac{q}{p}}\frac{\dd h}{h}\bigg)^{1/q}.
$$
Note that the concepts of atomic (wavelet or spline) expansions also carry over by restriction, that is, one only takes into account the atoms whose support has an overlap with $(0,1)$. In particular, every function $f \in B^s_{p,q}((0,1);E)$ with $0<s<1$ and $p,q \in [1,\infty]$ admits a Lip-atomic representation
\begin{equation}\label{Lip-atomic representation of Besov functions on unit interal}
  f(t) = \sum_{j=0}^\infty \sum_{m = 0}^{2^j} \lambda_{j,m}a_{j,m}(t),
\end{equation}
where $a_{j,m}$ are Lip-atoms according to Definition~\ref{def: Lip-atom} and $\lambda =(\lambda_{j,m})_{j\ge 0,m=0,\dots,2^j}$ satisfies that $\|\lambda | b_{p,q}\|<\infty$. Finally, if additionally $s > 1/p$, then $f \in  B^s_{p,q}((0,1);E)$ is a continuous function and therefore $f(0)$ and $f(1)$ are well-defined. In this case we will use the notation $ B^s_{p,q}([0,1];E)$ for the $E$-valued Besov spaces on $(0,1)$.

\subsection{Three equivalent norms}\label{subsec:vector-valued Besov spaces}

In this subsection we simplify the notation and write $\| \cdot\|_{B^s_{p,q}}$ for the Besov norms on $ B^s_{p,q}([0,1];E)$ instead of $\|\cdot \|_{B^s_{p,q}((0,1))}$. Moreover, from now on, we will always assume that $0<s<1$, $1\le p,q\le\infty$ and $s > 1/p$. 

Let us recall the following definitions of Besov spaces:
\begin{itemize}
  \item $B^s_{p,q}([0,1];E)$ is the space of all continuous functions~$f\colon[0,1] \to E $ such that   
        \[
          {\|f\|}_{B^s_{p,q}} := \|f(0)\| + \bigg(\int_{0}^1 \Big(\int_0^{1-h} \frac{\|f(x+h) - f(x)\|^p}{h^{sp}}\dd x\Big)^{\frac{q}{p}}\frac{\dd h}{h} \bigg)^{\frac{1}{q}}<\infty.
        \]
  \item $b_{p,q,(1)}^s([0,1];E) $ is the space of all continuous functions~$f\colon [0,1] \to E $ such that
        \[
          {\|f\|}_{b^s_{p,q},(1)} :=  {\|f(0)\|} +\bigg( \sum_{j \geq 0} 2^{jq(s - \frac{1}{p})} \Big(\sum_{m=0}^{2^j-1} {\Big \| f\big(\frac{m}{2^j}\big) - f\big(\frac{m+1}{2^j}\big) \Big \| }^p\Big)^{\frac{q}{p}}\bigg)^{\frac{1}{q}}<\infty.
        \]
  \item $b_{p,q,(2)}^s([0,1];E) $ is the space of all continuous functions~$ f\colon[0,1] \to E $ such that 
        \begin{align}\label{eq:Besov second order differences}
        \begin{split}
          &{\|f\|}_{b^s_{p,q},(2)} :=  {\|f(0)\|} + {\| f(1) - f(0) \|} \\ 
          & \quad\quad+\bigg( \sum_{j \geq 1} 2^{jq(s-\frac{1}{p})}  \Big(\sum_{m=0}^{2^{j-1}-1} {\Big \| f\big(\frac{m}{2^{j-1}}\big) - 2 f \big(\frac{2m+1}{2^j}\big) + f\big(\frac{m+1}{2^{j-1}}\big)  \Big \| }^p\Big)^{\frac{q}{p}}\bigg)^{\frac{1}{q}}<\infty .
        \end{split}
        \end{align}
\end{itemize}
For $p=\infty$ or $q=\infty$ the standard modifications are applied in all the three definitions. Note that all three norms define Banach space topologies on the set of functions where the respective norms are finite.

If $p=q$, then we use the short-writing $b^s_{p,(1)}$ and $b^s_{p,(2)}$ for $b^s_{p,p,(1)}$ and $b^s_{p,p,(2)}$, respectively. 

\begin{remark}
  Besov spaces are classically, i.e., in the setting of real valued functions, characterized by three overlapping but different expansion methods: atomic decompositions (see, e.g., \cite{Scharf2012} and the references therein), wavelet expansions (see, e.g., \cite{Triebel2004}), and by expansions with respect to piecewise harmonic functions (or, more generally, with respect to splines, see, e.g., \cite{Kamont1997}). All methods have their advantages and it is non-trivial to translate results into each other. In \cite{Scharf2012} it has been shown that for generic Banach spaces $E$ results for atomic decompositions also hold true (with $ E $-valued atoms) such as in the real valued case, in particular no UMD property of $E$ is needed. Translating results on expansion coefficients into results on transformed coefficients, for instance translating results which depend on second order differences into results on first order differences, has been directly performed in real-valued case in \cite{Kamont1997}. 
\end{remark}

In the following, we will verify that the three different definitions of Besov spaces indeed lead to the same function space. That is, we will show that for any Banach space~$E$, the above three norms are equivalent on $B^s_{p,q}([0,1];E)$. In particular, we will extend the result of Theorem~1 in \cite{Rosenbaum2009} from the scalar case to the vector valued case. Our proof scheme goes as follows: First we will prove the equivalence of these three norms for Sobolev spaces (which corresponds to the case $p = q$) in the next proposition, then by an interpolation argument we generalize the equivalence to all Besov spaces, i.e. allowing for $p \neq q$.

\begin{proposition}\label{prop: equivalence of three norms for Sobolev spaces}
  Let $E$ be a Banach space, $s\in (0,1)$ and $p\in [1,\infty]$ such that $s>1/p$. Then, the three above Besov norms are equivalent, i.e., 
  \begin{equation*}
    \|f\|_{B^s_{p,p}} \sim {\|f\|}_{b^s_{p,p},(1)}  \sim {\|f\|}_{b^s_{p,p},(2)} , \quad f\in B_{p,p}^s([0,1];E),
  \end{equation*}
  and, in particular, the three Banach spaces are isomorphic:
  \[
    B_{p,p}^s([0,1];E) = b_{p,(1)}^s([0,1];E) = b_{p,(2)}^s([0,1];E).
  \]
\end{proposition}

\begin{proof}
  The proof is done in three steps.\smallskip
  
  \noindent\textit{Step~1:} The first observation is that Rosenbaum's elementary proof for 
  \[
    b_{p,q,(1)}^s([0,1];E) = b_{p,q,(2)}^s([0,1];E)
  \]
  generalizes line by line to the vector valued case and even for general $p$ and $q$. For detailed arguments we refer to the proof of Theorem~1 on page 58--59 in \cite{Rosenbaum2009}. Hence, it is sufficient to show
  \[
    B_{p,p}^s([0,1];E) = b_{p,(2)}^s([0,1];E),
  \]
  which will be done in the next two steps. \smallskip
 
  \noindent\textit{Step~2:} We know by Proposition~\ref{prop: Lip-atomic representation of Besov functions} that vector valued Besov spaces on $\mathbb{R}$ allow for Lip-atomic characterizations of their norms, which by restriction to $[0,1]$ holds also on the unit interval, see Subsection~\ref{subsec:restriction to internval}.
  
  Now, we define for $i \in \N_0$ that $V_i := \{\frac{k}{2^i}:k = 0,\dots,2^i\}$. For every $\xi \in V_{i}\setminus V_{i-1}$ ($i\ge 1$) such that $\xi = \frac{2m + 1}{2^i}$ for some $m = 0,1,\dots,2^{i-1}-1$, we define a function $\psi^i_\xi(t)$ as follows:
  $$
    \psi^j_\xi(t) :=
    \begin{cases}
      2^i(t - \frac{m}{2^{i-1}}),& \text{if } \frac{m}{2^{i-1}} \le t < \xi, \\
      2^i(\frac{m+1}{2^{i-1}} - t),& \text{if } \xi \le t < \frac{m+1}{2^{i-1}}, \\
      0,& \text{otherwise}.
    \end{cases}
  $$
  Then one can verify that $2^{-i(s-1/p)}\psi^j_\xi(t)$, $i \ge 1$, $\xi \in V_i\setminus V_{i-1}$ are real-valued Lip-atoms according to Definition~\ref{def: Lip-atom}. Furthermore, for $i = 0$ we define
  $$
    \psi^0_0(t) := 1\quad \text{and} \quad \psi^0_1(t) := t
  $$
  for $t \in [0,1]$. Clearly, they are also real-valued Lip-atoms. 
  
  Next, fix an $f \in b^s_{p,(2)}([0,1];E)$. For $i \ge 1$ and $\xi = \frac{2m+1}{2^i} \in V_i \setminus V_{i-1}$ we define
  $$
    \lambda_{i,\xi} := 2^{i(s-\frac{1}{p})}\Big(-f(\frac{m+1}{2^{i-1}}) + 2f(\xi) - f(\frac{m}{2^{i-1}})\Big),
  $$
  and $\lambda_{0,0} := f(0)$, $\lambda_{0,1} := f(1)-f(0)$. Using the convention that $\frac{0}{0} = 0$, we can immediately see that
  $$
    a_{i,\xi}(t) := \frac{2^{-i(s-\frac{1}{p})}\psi^j_\xi(t)\lambda_{i,\xi}}{\|\lambda_{i,\xi}\|}
  $$
  for $i \ge 1$, $\xi \in V_i\setminus V_{i-1}$ and $a_{0,0}(t) = \frac{\psi^0_0(t)\lambda_{0,0}}{\|\lambda_{0,0}\|}$, $a_{0,1}(t) = \frac{\psi^0_1(t)\lambda_{0,1}}{\|\lambda_{0,1}\|}$ define a family of $E$-valued Lip-atoms in the sense of Definition~\ref{def: Lip-atom}. Since $f$ is continuous, $f(t)$ can be represented as:
  \begin{equation}\label{equation: Lip-atomic representation of f via Faber-Schauder basis}
    f(t) = \|\lambda_{0,0}\|a_{0,0}(t) + \|\lambda_{0,1}\|a_{0,1}(t) + \sum_{i = 1}^\infty\sum_{\xi \in V_i \setminus V_{i-1}}\|\lambda_{i,\xi}\|a_{i,\xi}(t),
  \end{equation}
  which in fact corresponds to the linear interpolation of~$f$ on dyadic points inside the interval~$[0,1]$ and therefore the convergence happens in $L^\infty([0,1];E)$ (and hence in $L^p([0,1];E)$ for all $1 \le p \le  \infty$). Moreover, it is straightforward to check that with $\lambda = (\|\lambda_{i,\xi}\|)$, 
  $$
    \|\lambda|b_{p,p}\| = \|f\|_{b^s_{p,(2)}} < \infty.
  $$
  Hence, by Proposition~\ref{prop: Lip-atomic representation of Besov functions} we can conclude that $f \in B^s_{p,p}([0,1];E)$ and, as the formula~\eqref{equation: Lip-atomic representation of f via Faber-Schauder basis} gives a special Lip-atomic representation of $f$, it holds that 
  \begin{equation}\label{equation: Sobolev norm is bounded by sequence norm}
    \|f\|_{B^s_{p,p}} \sim \inf \|\tilde{\lambda}| b_{p,p}\| \le \|\lambda | b_{p,p}\| = \|f\|_{b^s_{p,(2)}},
  \end{equation}
  where the infimum is taken over all admissible Lip-atomic representations of 
  $$
    f = \sum_{i=0}^\infty \sum_{m=0}^{2^i}\tilde{\lambda}_{i,m}\tilde{a}_{i,m}
  $$
  given as in the formula~\eqref{Lip-atomic representation of Besov functions on unit interal}.
  \smallskip
  
  \noindent\textit{Step~3:}
  In this step we will establish the converse of the inequality~\eqref{equation: Sobolev norm is bounded by sequence norm}, that is, we want to show that if~$f$ belongs to $B^s_{p,p}([0,1];E)$, then $f \in b^s_{p,(2)}([0,1];E)$ and 
  \begin{equation}\label{equation: sequence norm is bounded by Sobolev norm}
    \|f\|_{b^s_{p,(2)}} \lesssim \|f\|_{B^s_{p,p}},
  \end{equation}
  where the proportional constant is independent of $f$. Towards this end, we can assume $1<p<\infty$ and follow similar arguments as in the proof of Theorem~5.1 in the work of Kabanava~\cite{Kabanava2012}. For sake of brevity we will only state here the essential steps. 
  
  First, given a $f \in B^s_{p,p}([0,1];E)$, by Proposition~\ref{prop: Lip-atomic representation of Besov functions} we can find a Lip-atomic representation of~$f$:
  \[
    f = \sum_{j = 0}^\infty \sum_{m=0}^{2^j} \lambda_{j,m} a_{j,m}
  \] 
  whose (real) coefficients $ \lambda = (\lambda_{j,m})_{j\ge 0,m=0,\dots,2^{j}}$ satisfy
  \[
    \| \lambda | b_{p,p}\| \lesssim {\| f \|}_{B^s_{p,p}} \, .
  \]
  The idea is to expand the $E$-valued atoms $ a_{j,m} $ with respect to the Faber-Schauder basis $\psi^j_{\xi}(t)$ defined as in the Step~2. More precisely, we denote by $ c_{\xi}(f) $ the coefficient of the ($L^1$-normalized) basis element $ \psi^j_{\xi}$ centered at a dyadic $ \xi \in V_{j+1}\setminus V_{j}$ for $ j \geq 0 $, i.e.~$ c_0(f) = f(0) $, $ c_1(f) = f(1)-f(0) $,
  \[
    c_{(2m+1)/2^{j+1}}(f) = \Big( - f\big(\frac{m}{2^{j}}\big) + 2 f \big(\frac{2m+1}{2^{j+1}}\big) - f\big(\frac{m+1}{2^{j}}\big) \Big)
  \]
  for $m =0,\ldots,2^j $ and $ j \geq 0 $. Then since $s > 1/p$ one can follow the proof on page 196--197 in \cite{Kabanava2012} to obtain that
  \[
    c_{\xi}(f) : = \sum_{j=0}^\infty \sum_{m=0}^{2^j} \lambda_{j,m} c_{\xi}(a_{j,m})
  \]
  is well-defined due to uniform (and unconditional) convergence of $ \sum_{j=0}^\infty \sum_{m=0}^{2^j-1} \lambda_{j,m} a_{j,m}$ to~$f$. Furthermore, from the support and Lipschitz properties of Lip-atoms $a_{j,m}$ one can check that for every $i \in \N$ and every $\xi \in V_i \setminus V_{i-1}$,
  \begin{equation}\label{equation: estimate for atoms with j > i}
    \|c_{\xi}(a_{j,m}) \| \leq 2 \times 2^{-j(s-\frac{1}{p})}
  \end{equation}
  for all $j > i$, $m=0,\ldots,2^j$,
  \begin{equation}\label{equation: estimate for atoms with j < i}
    \| c_{\xi}(a_{j,m}) \| \leq 2^{-i} 2^{-j(s-1-\frac{1}{p})}
  \end{equation}
  for all $j\le i$, $m=0,\ldots,2^j$.
   Then, following~\cite{Kabanava2012} on page~197, we split
  \[
    c_{\xi}(f) = \Big( \sum_{j=0}^i + \sum_{j=i+1}^\infty \Big)(\sum_{m=0}^{2^{j}} \lambda_{j,m} c_{\xi}(a_{j,m})) =: x_{\xi}(f) + y_{\xi}(f).
  \]
  We estimate first
  \[
    X_{i,p,s}(f) := 2^{i(s-\frac{1}{p})}{\Big( \sum_{\xi \in V_i \setminus V_{i-1}} {\|x_{\xi}(f)\|}^p \Big)}^{\frac{1}{p}} \, .
  \]
  Due to the localization properties of the atoms $a_{j,m}$, we obtain that
  \begin{equation}\label{the cardinality of non-vanishing atoms}
    \# \{ \xi \in V_i \setminus V_{i-1} \, | \; c_{\xi} (a_{j,m}) \neq 0 \} \lesssim 2^{-(j-i)}
  \end{equation}
  for $ i \geq 1 $, $ i \ge j \geq 0 $.  Therefore, by exploiting~\eqref{equation: estimate for atoms with j < i} and \eqref{the cardinality of non-vanishing atoms} together with Minkowski's and H\"older's inequalities, we arrive at
  \begin{align*}
    \Big(\sum_{\xi \in V_i \setminus V_{i-1}} {\|x_{\xi}(f)\|}^p \Big)^{\frac{1}{p}}&\le \sum_{j=0}^i\Big( \sum_{\xi \in V_i\setminus V_{i-1}} \Big(\sum_{m=1}^{2^j}|\lambda_{j,m}|\|c_\xi(a_{j,m})\|\Big)^p \Big)^{\frac{1}{p}} \\
    & \le \sum_{j=0}^i \sum_{m=1}^{2^j} \Big(\sum_{\xi \in V_i\setminus V_{i-1}}|\lambda_{j,m}|^p\|c_\xi(a_{j,m})\|^p\Big)^{\frac{1}{p}} \\
    &\lesssim\sum_{j=0}^i\sum_{m=1}^{2^j}|\lambda_{j,m}|2^{-i}2^{-j(s-1-\frac{1}{p})}2^{-\frac{(j-i)}{p}} \\
    &\lesssim 2^{-i(1-\frac{1}{p})}\sum_{j=0}^i 2^{-j(s-1)}\Big(\sum_{m=1}^{2^j}|\lambda_{j,m}|^p\Big)^{\frac{1}{p}}.
  \end{align*}
  An application of Jensen's inequality leads to 
  \[
    {\Big ( \sum_{i \geq 1} X_{i,p,s}^p(f) \Big )}^{\frac{1}{p}} \lesssim \| \lambda| b_{p,p} \| \, .
  \]
  For more details regarding this estimate we refer readers to the proof of \cite[Theorem~5.1]{Kabanava2012}.
  
  Analogously we define
  \[
    Y_{i,p,s}(f) := 2^{i(s-\frac{1}{p})}{\Big( \sum_{\xi \in V_i \setminus V_{i-1}} {\|y_{\xi}(f)\|}^p \Big)}^{\frac{1}{p}} \, .
  \]
  This time we can use~\eqref{equation: estimate for atoms with j > i} and~\eqref{the cardinality of non-vanishing atoms} to obtain that 
  \[
    {\Big ( \sum_{i \geq 1} Y_{i,p,s}^p(f) \Big )}^{\frac{1}{p}} \lesssim \| \lambda | b_{p,p} \| \, .
  \]
  Therefore, by unconditional convergence of the respective series, we have
  \begin{align*}
    {\| f \|}_{b^s_{p,(2)}} &= \Big(\sum_{i = 1}^\infty 2^{i(s-\frac{1}{p})p}\sum_{\xi \in V_i \setminus V_{i-1}}\|c_\xi(f)\|^p\Big)^{\frac{1}{p}} \\
    &\leq {\Big ( \sum_{i \geq 1} X_{i,p,s}^p(f) \Big )}^{\frac{1}{p}} + {\Big ( \sum_{i \geq 1} Y_{i,p,s}^p(f) \Big )}^{\frac{1}{p}} \lesssim\| \lambda | b_{p,p}\| \lesssim {\| f \|}_{B^s_{p,p}},
  \end{align*}
  which gives the inequality~\eqref{equation: sequence norm is bounded by Sobolev norm}. Now a slight modification of the above proof leads to the corresponding result for the case $p = \infty$ and completes the proof.
\end{proof}

Now we turn to the general Besov spaces $B^s_{p,q}([0,1];E)$. If $E = \R$, then for fixed $s$ and $p$, we can identify the Besov space $B^s_{p,q}([0,1];\R)$ with the real interpolation space 
$$
  (B^{s_0}_{p,p}([0,1];\R), B^{s_1}_{p,p}([0,1];\R))_{\theta,q}
$$
for any $q \in [1,\infty]$, as long as $s_0 \neq s_1$ fulfills that $s = (1-\theta)s_0 + \theta s_1$ for some $\theta \in (0,1)$. We refer the reader to \cite[Section~2.4.2]{triebel2010} for more details about the real interpolation method applied to scalar valued Besov spaces. As pointed out in \cite[Section~3]{Amann2000}, the same properties hold also for vector valued Besov spaces, namely
$$
  (B^{s_0}_{p,p}([0,1];E), B^{s_1}_{p,p}([0,1];E))_{\theta,q} = B^s_{p,q}([0,1];E)
$$
for all $q \in [1,\infty]$ and $s_0 \neq s_1$ as above. This observation will allow us to derive the equivalences of the above three norms on all Besov spaces $B^s_{p,q}([0,1];E)$ from the corresponding results obtained in Proposition~\ref{prop: equivalence of three norms for Sobolev spaces} for Sobolev spaces. However, for this purpose we also need the following notions of sequence spaces:

\begin{definition}
  Let $X$ be a Banach space, $\sigma \in \R$ and $p \in (0,\infty]$. The Banach space $l^\sigma_p(X)$ is defined via
  $$
    l^\sigma_p(X):= \bigg\{\xi = (\xi_j)_{j \ge 0}\, :\, \xi_j \in X \text{ for all } j \ge 0,\, \|\xi\|_{l^\sigma_p(X)} := \Big(\sum_{j=0}^\infty (2^{\sigma j}\|\xi_j\|)^p\Big)^{\frac{1}{p}} < \infty\bigg \}.
  $$
\end{definition}

Note that the scalar counterpart of $l^\sigma_p$ was introduced by Triebel in \cite{Triebel1973}. For the later use we need the following lemma:

\begin{lemma}\label{lem: embedding result of interpolation weighted lp spaces}
  Let $\sigma_0,\sigma_1$ be two real numbers with $\sigma_0 \neq \sigma_1$, $p,q\in[1,\infty]$ and $X$ be a Banach space. Then, it holds that for all $\theta \in (0,1)$ that
  $$
    (l^{\sigma_0}_p(X),l^{\sigma_1}_p(X))_{\theta,q} \subset l^\sigma_q(X)
  $$ 
  with $\sigma = (1-\theta)\sigma_0 + \theta \sigma_1$, where $(l^{\sigma_0}_p(X),l^{\sigma_1}_p(X))_{\theta,q}$ denotes the real interpolation space of $l^{\sigma_0}_p(X)$ and $l^{\sigma_1}_p(X)$ with respect to $(\theta,q)$, and the inclusion ``$\subset$'' means that 
  $$
    \|\xi\|_{l^\sigma_q(X)} \lesssim \|\xi\|_{(l^{\sigma_0}_p(X),l^{\sigma_1}_p(X))_{\theta,q}}
  $$
  for all $\xi \in (l^{\sigma_0}_p(X),l^{\sigma_1}_p(X))_{\theta,q}$.
\end{lemma}

\begin{proof}
  The proof is inspired by Step~1 of the proof of \cite[Lemma~8.2.1]{Triebel1973}, where the scalar case $X = \R$ was considered. For the sake of notational simplicity we write $\| \cdot \|_{\theta,q}$ instead of $\| \cdot \|_{(l^{\sigma_0}_p(X),l^{\sigma_1}_p(X))_{\theta,q}}$ and we define $K(t;\xi) := K(t;\xi,l^{\sigma_0}_p(X),l^{\sigma_1}_p(X))$, which is called Peetre's $K$-functional, implicitly by
  $$
    \|\xi\|_{\theta,q} = \Big(\int_0^\infty t^{-\theta q} K^q(t;\xi)\, \frac{\d t}{t}\Big)^{\frac{1}{q}}
  $$
  for $\xi \in (l^{\sigma_0}_p(X),l^{\sigma_1}_p(X))_{\theta,q}$. Fix a $\xi = (\xi_j)_{j\ge 0}$ in $(l^{\sigma_0}_p(X),l^{\sigma_1}_p(X))_{\theta,q}$. Then we have for $1 \le p,q < \infty$ that
  \begin{align*}
    K^p(t;\xi) &= \inf_{\xi = \xi^0 + \xi^1; \xi^i \in l^{\sigma_i}_p(X), i=0,1}\{\|\xi^0\|_{l^{\sigma_0}_p} + t\|\xi^1\|_{l^{\sigma_1}_p}\}^p \\
    & \sim \inf_{\xi = \xi^0 + \xi^1; \xi^i \in l^{\sigma_i}_p(X), i=0,1}\{\|\xi^0\|_{l^{\sigma_0}_p}^p + t^p\|\xi^1\|_{l^{\sigma_1}_p}^p\} \\
    &= \inf_{\xi = \xi^0 + \xi^1; \xi^i \in l^{\sigma_i}_p(X), i=0,1}\Big\{\sum_{j=0}^\infty 2^{\sigma_0 j p}\|\xi^0_j\|^p + t^p\sum_{j=0}^\infty 2^{\sigma_1 j p}\|\xi^1_j\|^p\Big\} \\
    &\ge \inf_{\xi = \xi^0 + \xi^1; \xi^i \in l^{\sigma_i}_p(X), i=0,1}\Big\{\sum_{j=0}^\infty \min(2^{\sigma_0 j p},2^{\sigma_1 j p}t^p)\Big(\|\xi^0_j\|^p + \|\xi^1_j\|^p\Big)\Big\}\\
    &\gtrsim  \sum_{j=0}^\infty \min(2^{\sigma_0 j p},2^{\sigma_1 j p}t^p)\|\xi_j\|^p
  \end{align*}
  Because $\sigma_0 \neq \sigma_1$, we may assume that $\sigma_0 > \sigma_1$ and divide the interval $(0,\infty)$ into parts $[2^{(k-1)(\sigma_0 - \sigma_1)}, 2^{k(\sigma_0 - \sigma_1)})$, $k \in \mathbb{Z}$. As a consequence we have the following estimates:
  \begin{align*}
    \|\xi\|_{\theta,q}^q &= \int_0^\infty t^{-\theta q} K^q(t;\xi)\, \frac{\d t}{t} 
    \ge \int_1^\infty t^{-\theta q} K^q(t;\xi) \,\frac{\d t}{t} \\
    &\ge \sum_{k=1}^\infty 2^{-\theta q k(\sigma_0 - \sigma_1)}K^q(2^{(k-1)(\sigma_0 - \sigma_1)};\xi)2^{-k(\sigma_0 - \sigma_1)}(2^{k(\sigma_0 - \sigma_1)} -2^{(k-1)(\sigma_0 - \sigma_1)} ) \\
    &\gtrsim\sum_{k=0}^\infty 2^{-\theta q k(\sigma_0 - \sigma_1)}\Big(\sum_{j=0}^\infty \min(2^{\sigma_0 j p}, 2^{\sigma_1 j p + k(\sigma_0 -\sigma_1)p})\|\xi_j\|^p\Big)^{\frac{q}{p}}\\
    &\gtrsim \sum_{k=0}^\infty 2^{qk\sigma}\|\xi_k\|^q = C\|\xi\|^q_{l^\sigma_q},
  \end{align*}
  where we only consider the term with $j = k$ in the last inequality. A slight modification gives the desired embedding result also for the cases $p = \infty$ and/or $q = \infty$.
\end{proof}

\begin{remark}
  If $1<p,q<\infty$ and $X$ is a \textit{reflexive} Banach space, the weighted sequence spaces $l^{\sigma_0}_p(X)$ and $l^{\sigma_1}_p(X)$ are reflexive Banach spaces and therefore one can show that 
  $$
    (l^\sigma_q(X))^\prime = l^{-\sigma}_{q^\prime}(X^\prime)
  $$
  with $X^\prime$ being the dual space of $X$ and $\frac{1}{q} + \frac{1}{q^\prime} = 1$. Then, following the same lines as in the Step~3 of the proof for \cite[Lemma~8.2.1]{Triebel1973}, we can use Lemma~\ref{lem: embedding result of interpolation weighted lp spaces} to get that
  $$
    l^\sigma_q(X) = \Big(l^{-\sigma}_{q^\prime}(X^\prime)\Big)^\prime \subset \Big((l^{-\sigma_0}_{p^\prime}(X^\prime),l^{-\sigma_1}_{p^\prime}(X^\prime))_{\theta,q^\prime}\Big)^\prime = (l^{\sigma_0}_p(X),l^{\sigma_1}_p(X))_{\theta,q},
  $$
  and thus we even have $l^\sigma_q(X) = (l^{\sigma_0}_p(X),l^{\sigma_1}_p(X))_{\theta,q}$ in this case. However, remember that in this paper we are mainly working on the Banach spaces, as e.g. $E = c_0$, which are not reflexive, the spaces $L^p(D,\mu;E)$ fail to be reflexive as well for any $\sigma$-finite measure space $(D,\mu)$; as a result, for $X:= L^p(D,\mu;E)$ which will be used in the next theorem, we may \textit{not} have the equality $l^\sigma_q(X) = (l^{\sigma_0}_p(X),l^{\sigma_1}_p(X))_{\theta,q}$. Fortunately only the embedding result proved in Lemma~\ref{lem: embedding result of interpolation weighted lp spaces} will be needed for establishing our next theorem.
\end{remark}

Finally, we are able to show the equivalence of the three norms on general Besov spaces $B^s_{p,q}([0,1];E)$ for $1/p<s<1$ and $p,q \in [1,\infty]$.

\begin{theorem}\label{thm: equivalence of three norms on Besov spaces}
  Let $E$ be a Banach space. Let $s\in (0,1)$ and $p,q\in [1,\infty]$ such that $1/p < s < 1$.
  Then, the three Besov norms $\|\cdot\|_{B^s_{p,q}}$, ${\|\cdot\|}_{b^s_{p,q},(1)}$ and  ${\|\cdot\|}_{b^s_{p,q},(2)}$ are equivalent, i.e., 
  \begin{equation*}
    \|f\|_{B^s_{p,q}} \sim   {\|f\|}_{b^s_{p,q},(1)}  \sim {\|f\|}_{b^s_{p,q},(2)} , \quad f\in B_{p,q}^s([0,1];E),
  \end{equation*}
  and, in particular, the three Banach spaces are isomorphic:
  \[
    B_{p,q}^s([0,1];E) = b_{p,q,(1)}^s([0,1];E) = b_{p,q,(2)}^s([0,1];E).
  \]
\end{theorem}

\begin{proof}
  Thanks to Proposition~\ref{prop: equivalence of three norms for Sobolev spaces}, we may implicitly assume that $p \neq q$. Moreover, noting that Step~1 and Step~2 in the proof of Proposition~\ref{prop: equivalence of three norms for Sobolev spaces} remains valid for $p \neq q$, we only need to show that if $f \in B^s_{p,q}([0,1];E)$, then $f \in b^s_{p,q,(1)}([0,1];E)$ and
  \begin{equation}\label{eq: sequence norm is bounded by Besov norm}
    \|f\|_{b^s_{p,q,(1)}} \lesssim \|f\|_{B^s_{p,q}}.
  \end{equation}
  Now let us fix $s \in (0,1)$ and $ p \in (1,\infty]$ such that $s > 1/p$; and we pick $1 > s_0 > s_1 > 1/p$ such that $s = (1-\theta)s_0 + \theta s_1$ for some $\theta \in (0,1)$. Given a $f \in B^s_{p,q}([0,1];E)$, one has
  $$
    f \in (B^{s_0}_{p,p}([0,1];E), B^{s_1}_{p,p}([0,1];E))_{\theta,q} = B^s_{p,q}([0,1];E)
  $$
  due to the real interpolation argument for vector valued Besov spaces. By Proposition~\ref{prop: equivalence of three norms for Sobolev spaces} we can deduce further that
  \begin{equation}\label{eq: interpolation of sequence spaces}
    f \in (b^{s_0}_{p,(1)}([0,1];E), b^{s_1}_{p,(1)}([0,1];E))_{\theta,q}.
  \end{equation}
  Furthermore, let $D$ be the collection of all dyadic numbers in $[0,1]$ and $\mu$ be the counting measure on $D$ such that $\mu(\{a\}) = 1$ for all $a \in D$. It is fairly easy to see that for any $E$-valued continuous function $g$ defined on $[0,1]$ and for any $1 > r > \frac{1}{p}$, $1\le p,q \le \infty$, the statement that $g \in b^r_{p,q,(1)}([0,1];E)$ is equivalent to the condition that $\xi(g) = (\xi(g)_j)_{j \ge 0}$ belongs to $l^{\tau}_q(X)$, where $X = L^p(D,\mu;E)$, $\tau = r - \frac{1}{p}$ and for each $j \ge 1$, $\xi(g)_j$ is a $E$-valued mapping defined on $D$ such that for $a \in D$,
  \begin{align*}
    \xi(g)_j(a) := 
    \begin{cases}
      g(\frac{m}{2^j}) - g(\frac{m+1}{2^j}),& \text{if } a = \frac{m}{2^j} \text{ for } m = 0,\ldots,2^j - 1,\\
      0,& \text{otherwise},
    \end{cases}
  \end{align*}
  and for $j = 0$,
  \begin{align*}
    \xi(g)_0(a) := 
    \begin{cases}
      g(0),& \text{for } a = 0,\\
      g(1) - g(0),& \text{for } a = 1,\\
      0,& \text{otherwise}.  
    \end{cases}
  \end{align*}
  Moreover, it holds that $\|g\|_{b^r_{p,q},(1)} = \|\xi(g)\|_{l^\tau_q(X)}$. Using this identification, one can verify the following estimates for two Peetre's $K$-functionals which we are interested in:
  \begin{align*}
    K(t;f,b^{s_0}_{p,(1)},b^{s_1}_{p,(1)}) &= \inf_{f = f_0 + f_1, f_i\in b^{s_i}_{p,p,(1)}, i =0,1}\{\|f_0\|_{b^{s_0}_{p,(1)}} + t \|f_1\|_{b^{s_1}_{p,(1)}}\} \\
    &= \inf_{f = f_0 + f_1, f_i\in b^{s_i}_{p,p,(1)}, i =0,1}\{\|\xi(f_0)\|_{l^{\sigma_0}_p(X)} + t \|\xi(f_1)\|_{l^{\sigma_1}_p(X)}\} \\
    &\ge \inf_{\xi = \xi^0 + \xi^1; \xi^i \in l^{\sigma_i}_p(X), i=0,1} \{\|\xi^0\|_{l^{\sigma_0}_p(X)} + t \|\xi^1\|_{l^{\sigma_1}_p(X)}\} \\
    &=K(t;\xi(f),l^{\sigma_0}_p(X), l^{\sigma_1}_p(X))
  \end{align*}
  for $\sigma_0 := s_0 - 1/p$ and $\sigma_1 := s_1 - 1/p$, because every decomposition $f = f_0 + f_1$ with $f_0 \in b^{s_0}_{p,p,(1)}([0,1];E)$ and $f_1 \in b^{s_1}_{p,p,(1)}([0,1];E)$ corresponds to a decomposition of $\xi(f) = \xi(f_0) + \xi(f_1)$ with $\xi(f_0) \in l^{\sigma_0}_p(X)$ and $\xi(f_1) \in l^{\sigma_1}_p(X)$ due to the above construction of $\xi(f)$. Hence, the condition~\eqref{eq: interpolation of sequence spaces} ensures that 
  \begin{align*}
    \infty > \|f\|_{(B^{s_0}_{p,p}([0,1];E), B^{s_1}_{p,p}([0,1];E))_{\theta,q}} &\sim \Big(\int_0^\infty t^{-\theta q}K^q(t;f,b^{s_0}_{p,(1)},b^{s_1}_{p,(1)})\, \frac{\d t}{t} \Big)^{\frac{1}{q}} \\
    &\ge \Big(\int_0^\infty t^{-\theta q}K^q(t;\xi(f),l^{\sigma_0}_p(X), l^{\sigma_1}_p(X))\, \frac{\d t}{t} \Big)^{\frac{1}{q}} \\
    &=\|\xi(f)\|_{\theta,q}.
  \end{align*}
  Now, invoking Lemma~\ref{lem: embedding result of interpolation weighted lp spaces} and the fact that 
  $$
    B^s_{p,q}([0,1];E) = (B^{s_0}_{p,p}([0,1];E), B^{s_1}_{p,p}([0,1];E))_{\theta,q},
  $$
  we have 
  $$
    \infty > \|f\|_{B^s_{p,q}} \ge \|\xi(f)\|_{\theta,q}  \gtrsim  \|\xi(f)\|_{l^\sigma_q(X)} = \|f\|_{b^s_{p,q},(1)},
  $$
  where $\sigma = (1-\theta)\sigma_0 + \theta \sigma_1 = s - 1/p$. Clearly this shows that $f \in b^s_{p,q,(1)}([0,1];E)$ and the estimate~\eqref{eq: sequence norm is bounded by Besov norm} holds true.
\end{proof}

\begin{remark}\label{rmk:proof main theorem}
  By using a Lipschitz embedding of the metric space $(E,d)$ into a Banach space, see Subsection~\ref{subsec:embeddings of metric spaces into Banach spaces}, Theorem~\ref{thm:Besov characterization} follows immediately from Theorem~\ref{thm: equivalence of three norms on Besov spaces} above.
\end{remark}

\section{Embedding results between Besov and p-variation spaces}\label{sec:embedding results}

This section is devoted to provide embedding results between Besov and p-variation spaces. For this purpose we start by introducing the corresponding norms restricted to subintervals.

Let $(E,\|\cdot\|)$ be a Banach space, $s\in (0,1)$ and $p,q\in [1,\infty)$. We call $\mathcal{P}$ a partition of an interval $[s,t]\subset [0,1]$ if $\mathcal{P}=\{[t_i,t_{i+1}]\,:\, s=t_0 < t_1<\dots <t_n=t,\, n\in \mathbb{N}\}$. The restricted semi-norms are defined as follows:
\begin{itemize}
  \item For $f\in B^{s}_{p,q}([0,1];E)$ we set 
        \begin{equation*}
          \|f\|_{s,p,q;[s,t]}:= \bigg(\int_{0}^{t-s} \bigg(\int_{s}^{t-h} \frac{\|f(u+h)-f(u)\|^{p}}{h^{s p}}\dd u\bigg)^{\frac{q}{p}} \frac{\dd h}{h}\bigg)^{\frac{1}{q}}.
        \end{equation*}
  \item For $f\in C^{p\var}([0,1];E)$ we set 
        \begin{equation*}
          \|f\|_{p\var;[s,t]}:=\bigg(\sup_{\mathcal{P}\subset [s,t]} \sum_{[u,v]\in \mathcal{P}} \|f(v)-f(u)\|^p \bigg)^{\frac{1}{p}},
        \end{equation*}
        where the supremum is taken over all partitions $\mathcal{P}$ of the interval $[s,t]$.
\end{itemize}
Recall that $\|\cdot\|_{s,p,q}=\|\cdot\|_{s,p,q;[0,1]}$ and $\|\cdot\|_{p\var}=\|\cdot\|_{p\var;[0,1]}$.

For a continuous function $f\colon[0,1]\to E$ the property of finite $p$-variation is equivalent to the existence of a control function $\omega\colon \{(s,t)\,:\,0\leq s\leq t\leq T\}\to [0,\infty)$ such that 
\begin{equation*}
  \|f(t)-f(s) \|^p \leq \omega(s,t) ,\quad s,t\in [0,1]\text{ with } s<t,
\end{equation*}
see for example \cite[Section~1.2]{Lyons2007}. Let us recall that a function $\omega\colon \{(s,t)\,:\,0\leq s\leq t\leq T\}\to [0,\infty)$ is called control function if $\omega(s,s)=0$ for $s\in[0,T]$ and $\omega$ is super-additive, i.e. $ \omega (s,t) +\omega (t,u) \leq \omega (s,u)$ for $s\leq t\leq u$ in $[0,T]$. 

The next proposition presents the embedding of Besov spaces into $p$-variation spaces. It can be seen as a generalization of \cite[Theorem~2]{Friz2006}.

\begin{proposition}\label{prop:variation embeddings}
  Suppose that $(E,\|\cdot\|)$ is a Banach space. Let $s \in (0,1)$ and $p,q\in [1,\infty)$ be such that $s > 1/p$. Set 
  \begin{equation*}
    \alpha:= s-\frac{1}{p},\quad \beta:= \big(s + ((q^{-1}-p^{-1})\wedge 0)\big)^{-1} \quad \text{and}\quad \gamma:=\frac{1}{s-\epsilon},
  \end{equation*}
  for $\epsilon \in (0,s-1/p)$.
  \begin{enumerate}
   \item If $q \ge p$, then the $\beta$-variation of a function $f\in B^{s}_{p,q}([0,1];E)$ is controlled by a constant multiple of the control function 
         \begin{equation*}
           \omega (s,t):= \|f\|^{\beta}_{s,p,q;[s,t]} (t-s)^{\alpha \beta}
         \end{equation*}
         and thus $B^{s}_{p,q}([0,1];E) \subset C^{\beta\var}([0,1];E)$.
   \item If $q \ge p$, then the $\gamma$-variation of a function $f\in B^{s}_{p,q}([0,1];E)$ is controlled by a constant multiple of the control function 
         \begin{equation*}
           \omega (s,t):= \|f\|^{\gamma}_{\gamma,p,p;[s,t]} (t-s)^{\alpha \gamma}
         \end{equation*}
         and thus $B^{s}_{p,q}([0,1];E) \subset C^{\gamma\var}([0,1];E)$.
   \item If $q \le p$, then the $\beta$-variation of a function $f\in B^{s}_{p,q}([0,1];E)$  is controlled by a constant multiple of the control function 
         \begin{equation*}
           \omega (s,t):= \|f\|^{\beta}_{s,p,p;[s,t]} (t-s)^{\alpha \beta},
         \end{equation*}
         and thus $B^{s}_{p,q}([0,1];E) \subset C^{\beta\var}([0,1];E)$.
  \end{enumerate}
  In particular, for $f\in B^{s}_{p,q}([0,1];E)$ one has 
  \begin{equation*}
    \|f\|_{\beta\var}\lesssim  \|f\|_{s,p,q}.
  \end{equation*}
\end{proposition}

\begin{remark}
  While the embedding result stated in (1) is sharper than the one in (2) in the case $p=q$, the converse might be true in general depending on the choice of $\epsilon$.
\end{remark}

\begin{proof}
  Since $s> 1/p$ and $p\in (1,\infty)$, classical Besov embedding, see e.g. \cite[Corollary~26]{Simon1990}, leads to 
  \begin{equation*}
    \|f(t)-f(s)\| \leq \sup_{|t-s|\geq h >0 } \bigg(\frac{\|f(s+h)-f(s)\|}{h^{s - \frac{1}{p}}} \bigg) |t-s|^{s - \frac{1}{p}}\lesssim \|f\|_{s,p,q;[s,t]} |t-s|^{s - \frac{1}{p}},
  \end{equation*}
  for every $f\in B^{s}_{p,q}([0,1];E)$, that is $B^{s}_{p,q}([0,1];E)\subset B^{s-1/p}_{\infty,\infty}([0,1];E)= C^{\alpha}([0,1];E)$.
  
  (1) To see the finite $\beta$-variation of $f$ in the case $q\ge p$, we observe that the control function $\omega$ can be written as
  \begin{equation*}
    \omega (s,t)=  \tilde \omega (s,t)^\frac{\beta}{q} |t-s|^{\alpha \beta}\quad \text{with}\quad  \tilde \omega (s,t):=  \|f\|^{q}_{s,p,q;[s,t]}.
  \end{equation*}
  Because $\hat \omega (s,t):=|t-s|$ is a control function and $\beta/q +\alpha \beta=1$, it is sufficient to verify that $\tilde \omega $ is also a control function by \cite[Exercise~1.9]{Friz2010}. Thanks to the integral definition the continuity of $\tilde \omega$ is obvious and for the super-additivity we use $q \ge p$ to get 
  \begin{align*}
    \tilde \omega (s,t)+\tilde \omega (t,u) 
    & = \int_{0}^{t-s} \bigg(\int_{s}^{t-h} \frac{\|f(x+h)-f(x)\|^{p}}{h^{s p}}\dd x\bigg)^{\frac{q}{p}} \frac{\dd h}{h}\\
    &\qquad\qquad \quad  + \int_{0}^{u-t} \bigg(\int_{t}^{u-h} \frac{\|f(x+h)-f(x)\|^{p}}{h^{s p}}\dd x\bigg)^{\frac{q}{p}} \frac{\dd h}{h}\\
    & \leq \int_{0}^{u-s} \bigg(\int_{s}^{t}\1_{[s,t]}(x+h) \frac{\|f(x+h)-f(x)\|^{p}}{h^{s p}}\dd x\bigg)^{\frac{q}{p}} \\
    & \qquad\qquad \quad   +  \bigg(\int_{t}^{u}  \1_{[t,u]}(x+h)\frac{\|f(x+h)-f(x)\|^{p}}{h^{s p}}\dd x\bigg)^{\frac{q}{p}} \frac{\dd h}{h}\\
    & \leq \int_{0}^{u-s}  \bigg(\int_{s}^{u} (\1_{[s,t]}(x+h)+\1_{[t,u]}(x+h)) \frac{\|f(x+h)-f(x)\|^{p}}{h^{s p}}\dd x\bigg)^{\frac{q}{p}} \frac{\dd h}{h}\\
    &=\tilde \omega (s,u)
  \end{align*}
  for $0\leq s \leq t\leq u \leq 1$. In particular, we have proven that $B^{s}_{p,q} ([0,1];E)\subset C^{\beta\var}([0,1];E)$. 
  
  (2) By Besov embedding, see e.g. \cite[Corollary~15]{Simon1990}, we obtain 
  \begin{equation*}
    \|f\|_{\gamma,p,p;[s,t]}\lesssim \|f\|_{s,p,q;[s,t]}.
  \end{equation*}
  Therefore, (2) follows by the same arguments as for (1) or as an direct application of \cite[Theorem~2]{Friz2006}.
  
  (3) For $q \le p$ we have $\beta = s^{-1}$ and
  \begin{equation*}
     \|f\|_{s,p,p;[s,t]}\lesssim \|f\|_{s,p,q;[s,t]},
  \end{equation*}
  see e.g. \cite[Theorem~11]{Simon1990}. Hence, (3) follows again by the same arguments as for (1) or directly by \cite[Theorem~2]{Friz2006}.
\end{proof}

The reverse embedding, that is, how $p$-variation spaces embed into Besov spaces, is formulated in the next proposition. 

\begin{proposition}\label{prop:besov embedding}
   Suppose that $(E,\|\cdot\|)$ is a Banach space and $\beta \in (1,\infty)$. Then, the inclusion
  \begin{equation*}
    C^{\beta\var}([0,1];E) \subset B^{1/\beta}_{\beta,\infty}([0,1];E)
  \end{equation*}
  holds and for $f\in C([0,1];E)$ one has 
  \begin{equation*}
    \|f\|_{1/\beta,\beta,\infty} \leq 2 \|f\|_{\beta\var}.
  \end{equation*}
\end{proposition}

\begin{proof}
  For $f\in  C^{\beta\var}([0,1];E)$ we need to show that
  \begin{equation*}
    \|f\|_{1/\beta,\beta,\infty}^\beta:= \sup_{1> h>0} h^{-1} \int_{0}^{1-h} \| f(x+h)-f(x) \|^{\beta} \dd x <\infty,
  \end{equation*}
  that is, it is sufficient to show that 
  \begin{equation*}
     \int_{0}^{1-h} \| f(x+h)-f(x) \|^{\beta} \dd x \lesssim h
  \end{equation*}
  for every $h\in (0,1)$.
  
  Indeed, given $h\in (0,1)$, take a partition $0=x_0<x_1<\dots<x_N < 1-h$ such that 
  \begin{equation*}
    |x_{k+1}-x_k|= h \text{ for } k=0,\dots,N-1\quad \text{and} \quad |x_N-1+h|\leq h.
  \end{equation*}
  Setting $x_{N+1}:=1-h$, we observe that
  \begin{align*}
    & \int_{0}^{1-h} \| f(x+h)-f(x) \|^{\beta}\dd x 
    = \sum_{k=0}^{N}  \int_{x_k}^{x_{k+1}} \| f(x+h)-f(x) \|^{\beta}\dd x \\
    &\quad= \sum_{k=0}^{N-1}\int_{0}^h \|f(x_{k+1}+y)-f(x_k+y)\|^{\beta} \dd y + \int_{0}^{x_{N+1}-x_N}\|f(x_{k+1}+y)-f(x_k+y)\|^{\beta} \dd y.
  \end{align*}
  Since for $y\in [0,h]$ and $k=0,\dots,N$
  \begin{equation*}
    \|f(x_{k+1}+y)-f(x_k+y)\|  \leq \|f\|_{\beta\var;[x_{k},x_{k+2}]}
  \end{equation*}
  with $x_{N+2}:=1$, we obtain
  \begin{align*}
    \int_{0}^{1-h} \| f(x+h)-f(x) \|^{\beta}\dd x \leq 2 \|f\|_{\beta\var;[0,1]}^\beta h.
  \end{align*}
  Re-arranging and taking the superemum over~$h$ gives 
  \begin{equation*}
    \sup_{1>h>0} h^{-1} \int_{0}^{1-h} \| f(x+h)-f(x) \|^{\beta}\dd x \leq 2 \|f\|_{\beta\var;[0,1]}^\beta,
  \end{equation*}
  which completes the proof. 
\end{proof}

\providecommand{\bysame}{\leavevmode\hbox to3em{\hrulefill}\thinspace}
\providecommand{\MR}{\relax\ifhmode\unskip\space\fi MR }
\providecommand{\MRhref}[2]{%
  \href{http://www.ams.org/mathscinet-getitem?mr=#1}{#2}
}
\providecommand{\href}[2]{#2}


\end{document}